\documentclass[11pt,reqno]{amsart}
\usepackage{iftex}
\ifXeTeX
  \usepackage{fontspec}
\else
  \usepackage[utf8]{inputenc}
  \usepackage[T2A]{fontenc}
\fi
\usepackage[russian,english]{babel}

\usepackage{caption}
\usepackage{amsmath}
\usepackage{amssymb}
\usepackage{amsfonts}
\usepackage{graphicx}
\usepackage{xcolor}
\usepackage{longtable}
\usepackage{array}
\usepackage{booktabs}
\usepackage{lastpage}
\usepackage[colorlinks]{hyperref}
\usepackage{tikz}

\usepackage{titlesec}
\titlespacing{\section}{0pt}{2.5ex}{1.5ex}
\titlespacing{\subsection}{0pt}{1.5ex}{1ex}
\titlespacing{\subsubsection}{0pt}{1.5ex}{1ex}
\titleformat{\section}{\large\bfseries\centering}{\thesection}{1em}{}
\titleformat{\subsection}[runin]{\bfseries}{\thesubsection.}{0.5em}{}[.\mbox{\ }]
\titleformat{\subsubsection}[runin]{\bfseries}{\thesubsubsection.}{0.4em}{}[.\mbox{\ }]

\newsavebox{\fitboxbox}
\newenvironment{fitbox}{\begin{lrbox}{\fitboxbox}}{\end{lrbox}%
  \ifdim\wd\fitboxbox>\linewidth\resizebox{\linewidth}{!}{\usebox{\fitboxbox}}\else\usebox{\fitboxbox}\fi}

\newcommand{\Z}{\mathbb{Z}}
\newcommand{\Cay}{\operatorname{Cay}}

\newcommand\orcidicon[1]{\href{https://orcid.org/#1}{\includegraphics[scale=0.02]{orcid.pdf}}}

\def\firstpage{1} 
\newtheorem{lemma}{Lemma}
\newtheorem{theorem}{Theorem}
\newtheorem{corollary}{Corollary}
\newtheorem{proposition}{Proposition}
\theoremstyle{definition}

\theoremstyle{remark}
\newtheorem{remark}{Remark}

\hypersetup{linkcolor=blue,filecolor=black,urlcolor=blue, citecolor=blue}

\begin{document}
\renewcommand{\refname}{References}
\renewcommand{\proofname}{Proof}
\renewcommand{\figurename}{Fig.}

\thispagestyle{empty}

\title[CRCs in four-dimensional lattice graphs]{\large Completely regular codes in four-dimensional lattice graphs}

\author[I.YU. MOGILNYKH]{{\bf I.Yu. Mogilnykh}}

\address{Ivan Yurevich Mogilnykh
\newline\hphantom{iii} Sobolev Institute of Mathematics,
\newline\hphantom{iii} pr. Koptyuga, 4,
\newline\hphantom{iii} 630090, Novosibirsk, Russia}%
\email{\textcolor{blue}{ivmog84@gmail.com}}%

\thanks{\sc Mogilnykh, I.Yu., Completely regular codes in four-dimensional lattice graphs.}
\thanks{\copyright \ 2026 Mogilnykh I.Yu.}
\thanks{\rm The work was carried out within the state assignment of the Sobolev Institute of Mathematics SB RAS, project no. FWNF-2026-0011.}
\thanks{\it Received January, 1, 2026, Published December, 31, 2026}%

\vspace{1cm}
\maketitle {\small

\vspace{-12pt}

\bigskip
\bigskip

\begin{quote}
\noindent{\bf Abstract:} We obtain series of new completely regular codes in Cayley graphs of the groups $\Z_q^n$ with the Lee metric. Combined with binary linear programming methods, the constructions lead to enumeration of all completely regular codes in such graphs for the groups $\Z_q^4$, $q\leq 7$ up to parameter matrices.
 \medskip

\noindent{\bf Keywords:} completely regular code, perfect coloring, covering, bipartite graph, infinite lattice, Lee metric, Hamming metric.
 \end{quote}
}

\bigskip

\section{Introduction}

 The undirected {\it Cayley graph} $\Cay(G,S)$ of an abelian group $G$  with a generating set $S$ such that $0\notin S$, $-S=S$, is the graph with vertex set $G$ whose edges are the pairs $\{x,x+a\}$, $x\in G$, $a\in S$. The graphs considered in this paper are Cayley graphs  of abelian groups. By $e^i=(0,\ldots,0,1,0,\ldots,0)$ we denote the word with the one in the $i$th position and all other coordinates equal to zero. Below we define the classes of graphs used in the paper.
\begin{itemize}
\item $H(n,q)=\Cay(\Z_q^n,\{\alpha e^i:\ 1\le i\le n,\ \alpha\in\Z_q\setminus\{0\}\})$ is the Hamming graph;
\item $G_q^n=\Cay(\Z_q^n,\{\pm e^i:\ 1\le i\le n\})$ and $G^n=\Cay(\Z^n,\{\pm e^i:\ 1\le i\le n\})$ are the lattices with the Lee metric ($l_1$-metric);
\item $D_q^n=\Cay(\Z_q^n,\{0,1,-1\}^n\setminus\{0^n\})$ and $D^n=\Cay(\Z^n,\{0,1,-1\}^n\setminus\{0^n\})$ are the lattices with the $l_\infty$-metric.
\end{itemize}
For $q\ge3$ the graph $G_q^n$ is regular of degree $2n$, and the graph $D_q^n$ is regular of degree $3^n-1$; $G_4^n$ is isomorphic to $H(2n,2)$ (the cycle $G_4$ is the square $H(2,2)$), $G_3^n=H(n,3)$, $G_2^n=H(n,2)$. For even $q$ the graph $G_q^n$ is bipartite: the parts are formed by the words with even and with odd Lee weight.

The notion of a completely regular code arose in the 1970s
in connection with the solution of the problem of the existence of perfect codes
with arbitrary covering radius in the Hamming metric
over field alphabets \cite{Z1},\cite{Z2},\cite{T}.
Perfect codes in the Lee metric have been studied since
the seminal paper of Golomb and Welch \cite{GW70}. One of the classical problems for completely regular codes
is the classification of their parameters in a given family of graphs.
In particular, classifications of the parameters of all
completely regular codes (CRCs) in the Hamming graphs
$H(n,2)$ were obtained ~\cite[Section~7]{CRC} for $n\leq 13$ (in $G_{4}^n$, $n\leq 6$ respectively)
 and in the Hamming graphs $H(n,3)$ for $n\leq 6$.

For the Lee metric graphs, the problem of classifying CRCs
has been considered in several papers. In the graph of the infinite square
lattice $G^2$, the case of covering radius~$1$ was studied
in \cite{Ax03}, and the case of covering radius~$2$
by Puzynina in \cite{Puz05}. For covering radius~$1$,
several restrictions on the parameter matrix were obtained in \cite{Dobrec}.
A complete classification of CRCs
of arbitrary covering radius in $G^2$ was obtained
in \cite{ASV12}. Thus, the first open
cases are the lattices $G^3$ and $G^4$.

In \cite{MV_Manh}, CRCs in $G^n$
whose first two diagonal elements of the parameter matrix
are equal to zero were studied. In particular, all such codes
in the graph $G^3$ were enumerated. We also mention an important paper \cite{AvgVas22},
in which, for CRCs in the graphs $G_q^n$ and $G^n$,
the monotonicity of the superdiagonal and subdiagonal
elements of the parameter matrix was established, as well as the upper bound~$2n$
on the covering radius.

The main definitions concerning completely regular codes and the basic constructions are given in Section~\ref{S2}. In Section~\ref{S3} we propose new series of CRCs with covering radius $2$ in the graphs $G_{6r}^{\frac{3^l-1}{2}}$ (two constructions) and $G_{6r}^{2^l}$ for any $l$, $r\geq 1$, and in the graph $G_{q}^{2m+1}$ for even $q\geq 4$ if there exists a perfect code in $G_q^m$. Mostly, the constructions are based on the composite nature of the alphabet size $6$ and combine binary or ternary perfect codes with special symbol functions or with classical methods such as the Solov'eva construction \cite{Sol}. Moreover, the new codes are not obtained from binary and ternary codes of small length in the Hamming metric by an inverse homomorphism, which follows from the enumeration of the parameters of these CRCs \cite[Section 7]{CRC}. In the final Section~\ref{S4} we list
 the parameter matrices of all CRCs in $G_q^n$, $n\leq 4$, $q\leq 7$, obtained using binary linear programming; each matrix is confirmed by a construction of a CRC.

\section{Basic definitions and results}\label{S2}
\subsection{Perfect colorings and codes in graphs}

\paragraph{Perfect colorings and CRCs} A partition $({\mathcal C}_0,\ldots,{\mathcal C}_{k-1})$ of the vertex set of a graph is called a \emph{perfect coloring} with \emph{parameter matrix} $M=(M_{ij})_{i,j=0}^{k-1}$ if every vertex of ${\mathcal C}_i$ has exactly $M_{ij}$ neighbors in ${\mathcal C}_j$. For a code (a nonempty set of vertices) ${\mathcal C}$ we put ${\mathcal C}_i$ as the set of vertices at distance $i$ from ${\mathcal C}$ in the graph. The number $\rho=\max\{i: {\mathcal C}_i\neq\varnothing\}$ is the \emph{covering radius} of the code ${\mathcal C}$. A code ${\mathcal C}$ is called \emph{completely regular} (briefly, a CRC) if the partition $({\mathcal C}_0,\ldots,{\mathcal C}_\rho)$ is a perfect coloring. The parameter matrix of a CRC is tridiagonal (i.e. $M_{ij}=0$ for $|i-j|\ge2$) and if the graph is $\delta$-regular, it is recovered from the \emph{intersection array}
\[
[b_0,\ldots,b_{\rho-1};\,c_1,\ldots,c_\rho],\qquad b_i=M_{i,i+1},\quad c_i=M_{i,i-1},\quad M_{ii}=\delta-b_i-c_i.
\]

The intersection array is mostly used  for the compact notation of CRCs in the tables of Section~4 and in the rest of the text CRCs are represented by their parameter matrices $M$. A classic example of CRC is a one-vertex code in a distance-regular graph and its intersection array coincides with the intersection array of the graph. For a one-vertex code in the Hamming graph $H(n,q)$ we have $c_i=i$, $1\le i\le n$, $b_i=(n-i)(q-1)$, $0\le i\le n-1$.

{\it A perfect code} (one-error-correcting perfect code) in a $\delta$-regular graph is a CRC with intersection array $[\delta;\,1]$, i.e. with the matrix $\begin{pmatrix}0&\delta\\1&\delta-1\end{pmatrix}$; thus, for $n=1$ the singleton in $H(1,q)=K_q$ from the example above is a perfect code with the matrix $\begin{pmatrix}0&q-1\\1&q-2\end{pmatrix}$. Such codes exist in $G^n$ for all $n\ge1$ \cite{GW70}. More generally, for $1\le k\le\delta$ a \emph{$k$-fold perfect code} in a $\delta$-regular graph is a CRC with parameter matrix $\begin{pmatrix}k-1&\delta+1-k\\k&\delta-k\end{pmatrix}$, i.e. a code meeting every ball of radius $1$ in exactly $k$ vertices; a $1$-fold perfect code is a perfect code.
\paragraph{Linear codes.} A code ${\mathcal C}\subseteq\Z_q^n$ is called \emph{linear} if it is a subgroup of the group $\Z_q^n$. Every linear code is given by a \emph{parity-check matrix}, i.e. a matrix $H$ over $\Z_q$ with columns $h_1,\ldots,h_n$ such that ${\mathcal C}=\{x\in\Z_q^n:\ Hx^\top=0\}$. For Theorems~\ref{T3}--\ref{T4}, in the case when all input codes are linear, we briefly indicate the description of the resulting code in terms of the parity-check matrices of the input codes. We note that there are also nonlinear perfect codes in Lee metric \cite{Et11},  \cite[Remarks 2, 4]{MS23}.

\subsection{Coverings}\label{SS:cover}
A surjective map $f$ from the vertices of a graph $\Gamma'$ onto the vertices of a graph $\Gamma$ is called a $\lambda$-fold {\it covering} if for every vertex $v'$ of $\Gamma'$ the map $f$ sends the neighbors of $v'$ to neighbors of $f(v')$ and every neighbor of $f(v')$ has exactly $\lambda$ preimages among the neighbors of $v'$ (for $\lambda=1$ this is the usual notion of a covering). For Cayley graphs, the homomorphisms described in Proposition~\ref{P:cover} are $\lambda$-fold coverings.

\begin{proposition}\label{P:cover} \cite[Lemma 4.1]{BKMTV21}
Let $\phi:G'\to G$ be a surjective  homomorphism of a group $G'$ to a group $G$, and let $S'\subseteq G'$ and $S\subseteq G$ be such that $0\notin S,S'$, $-S=S$, $-S'=S'$, $\phi(S')=S$ and every element of $S$ has exactly $\lambda$, $\lambda \ge1$, preimages in $S'$. If ${\mathcal C}$ is a CRC in $\Cay(G,S)$ with parameter matrix $M$, then $\phi^{-1}({\mathcal C})$ is a CRC in $\Cay(G',S')$ with parameter matrix $\lambda M$.
\end{proposition}

 A homomorphism $\Z_q^n\to\Z_q^m$ has the form $x\mapsto xT$ for an $n\times m$ matrix $T$; it is surjective, for example, if some minor of order $m$ of the matrix $T$ is coprime to $q$. Let $\phi_q:\Z_{qr}^n\to\Z_q^n$ (or $\Z^n\to\Z_q^n$) be the coordinatewise reduction modulo $q$:
 $$\phi_q(x_1,\ldots,x_n)=(x_1\mbox{ mod }q,\ldots,x_n\mbox{ mod }q). $$ Consider several basic constructions for CRCs arising from group homomorphisms (see also \cite{MV_Manh}).

\noindent\textbf{ K1.} Let $r$ be any integer, $r\ge2$ and ${\mathcal C}$ be a CRC in $G_2^n=H(n,2)$ with parameter matrix $M$. Then $\phi_2^{-1}({\mathcal C})$ is a CRC in $G_{2r}^n$ (respectively, in $G^n$) with matrix $2M$ (here $\lambda=2$, since $e^i=-e^i$ in $\Z_2^n$).

\noindent\textbf{ K2.} Let $r$ be any integer, $r\ge1$, $q\ge3$ and ${\mathcal C}$ be a CRC in $G_q^n$ with parameter matrix $M$. Then $\phi_q^{-1}({\mathcal C})$ is a CRC in $G_{qr}^n$ (respectively, in $G^n$) with matrix $M$ (here $\lambda=1$).

\noindent\textbf{ K3.} Let $q$ be odd, $N=\frac{m(q-1)}{2}$, and let the rows of the $N\times m$ matrix $T$ be all the words $\alpha e^j$, $1\le j\le m$, $1\le\alpha\le\frac{q-1}{2}$. Then $\phi(x)=xT$ maps $\{\pm e^i:1\le i\le N\}$ bijectively onto $\{\alpha e^j:\ 1\le j\le m,\ \alpha\in\Z_q\setminus\{0\}\}$. Consequently, if ${\mathcal C}$ is a CRC in $H(m,q)$ with matrix $M$, then $\phi^{-1}({\mathcal C})$ is a CRC in $G_q^{N}$ with matrix $M$.

\noindent\textbf{ K4.} Let $q\ge3$, $N=\frac{3^m-1}{2}$, and let the rows of the $N\times m$ matrix $T$ be all the words of length $m$ over $\{0,1,-1\}$ whose first nonzero symbol is $1$. Then $\phi(x)=xT$ maps $\{\pm e^i:1\le i\le N\}$ bijectively onto $\{0,1,-1\}^m\setminus\{0^m\}$. Consequently, if ${\mathcal C}$ is a CRC in $D_q^m$ with matrix $M$, then $\phi^{-1}({\mathcal C})$ is a CRC in $G_q^{N}$ with matrix $M$.

\subsection{Basic CRCs, addition and multiplication}\label{SS:simple}
We regard the following codes as the basic, simplest sources of CRCs.
\begin{itemize}
\item A one-vertex code in a distance-regular graph, for example in the cycle $G_q$ or in a Hamming graph;
\item Perfect codes in the Hamming metric and in the Lee metric. In particular, a binary perfect code of length $n$ in $H(n,2)$ exists if and only if $n=2^m-1$, $m\ge1$ (necessity follows from the Hamming bound, sufficiency from the Hamming codes). If $(2n+1)\mid q$, then the code $\{x\in\Z_q^n:\ \sum_{i=1}^n i x_i\equiv 0 \pmod{2n+1}\}$ is perfect in $G_q^n$ \cite{AHM}.
\item Extended binary perfect codes. Appending a parity check to each codeword gives an extended perfect code of length $n=2^m$ and cardinality $2^{n-m-1}$ with minimum distance $4$ (for $m\ge2$). For $m\ge2$ it is known \cite[Proposition 1.38]{CRC} that this code is a CRC in $H(n,2)=G_2^n$ with covering radius $\rho=2$ and parameter matrix
\[
M=\begin{pmatrix}
0&n&0\\
1&0&n-1\\
0&n&0
\end{pmatrix}.
\]
 For example, for $m=2$ the extended Hamming code $\{0000,1111\}$ is a CRC in $H(4,2)=G_2^4$ with array $[4,3;\,1,4]$, and for $m=3$ the extended perfect code ${\mathcal C}$ of length $8$ is a CRC in $H(8,2)=G_2^8$ with array $[8,7;\,1,8]$. By Construction~K1, for $r\ge2$ the preimage $\phi_2^{-1}({\mathcal C})$ is a CRC in $G_{2r}^n$ (and in $G^n$) with matrix $2M$, i.e. with array $[2n,\,2n-2;\,2,\,2n]$; for example, for $n=4$ we obtain a CRC in $G_6^4$ with array $[8,6;\,2,8]$.

 Note that recently, in \cite{K2CRC}, a criterion for the extension of CRC code  to be CRC was obtained in terms of the parameter matrices in a wide family of bipartite graphs, including the Hamming graphs over binary alphabets and the Lee graphs over alphabets of even size.

\item For even $q\ge4$ each part of the graph $G_q^n$ is a CRC with matrix $\begin{pmatrix}0&2n\\2n&0\end{pmatrix}$ (the \emph{even-weight code});
\item Even-weight subcodes of perfect codes. Let $P$ be a perfect code in a  bipartite $\delta$-regular graph with parts $V_0,V_1$ that meets both parts. Then (see e.g., \cite[Proposition 1.32]{CRC}) $P\cap V_0$  is a CRC with parameter matrix
\[
\begin{pmatrix}
0&\delta&0&0\\
1&0&\delta-1&0\\
0&\delta-1&0&1\\
0&0&\delta&0
\end{pmatrix}.
\]

\item $k$-fold perfect codes are CRCs with parameter matrix \[\begin{pmatrix}
k-1&\delta-k+1\\
k&\delta-k\\
\end{pmatrix}.
\]In particular, if $P$ is a perfect code in a Cayley graph $\Cay(G,S)$, $|S|=\delta$, that is a subgroup of the group $G$, and the group $G$ is abelian, then the union of any $k$ cosets of $P$, $1\le k\le\delta$, is a $k$-fold perfect code. As indicated at the beginning of this section, if $(2n+1)\mid q$, then there exists a perfect code of the graph $G_q^n$ that is a subgroup of $\Z_q^n$, and therefore for every $k$, $1\le k\le 2n$, there exists a $k$-fold perfect code in $G_q^n$.

\end{itemize}
The following two operations, see \cite{FDF} provide new CRCs from existing ones. Let $q\ge3$ and ${\mathcal C}$ be a CRC in $G_q^n$ with parameter matrix $M$.
\begin{itemize}
\item \emph{Addition} ($+2k$): ${\mathcal C}\times\Z_q^k$ is a CRC in $G_q^{n+k}$ with matrix $M+2kI$, where $I$ is the identity matrix.
\item \emph{Multiplication} ($*k$): the preimage of ${\mathcal C}$ under the homomorphism $\phi:\Z_q^{kn}\to\Z_q^n$,
\[
\phi(x_1,\ldots,x_{kn})=\Bigl(\sum_{t=0}^{k-1} x_{1+tn},\ \sum_{t=0}^{k-1} x_{2+tn},\ \ldots,\ \sum_{t=0}^{k-1}x_{n+tn}\Bigr),
\]
is a CRC in $G_q^{kn}$ with matrix $kM$.
\end{itemize}
In the tables of Section~4, the notation $G_q^n+2k$ means a CRC in $G_q^{n+k}$ obtained by addition from a CRC in $G_q^n$, and the notation $G_q^n*k$ means a CRC obtained by multiplication by $k$ from a CRC in $G_q^n$; the operations can be combined, for example $G_q^n*k+2l$. In all cases $G_q^n$ is the graph in which the original code is constructed (the construction of the original code is indicated in the table for $G_q^n$; its intersection array does not change under addition, and is multiplied by $k$ under multiplication).
\subsection{A criterion for CRCs with covering radius 2}
All the new codes of Section~3 are obtained using the following simple criterion.
\begin{proposition}\label{P1}
Let $\Gamma$ be a $\delta$-regular graph, and let $T={\mathcal C}\cup\widetilde{{\mathcal C}}$, where ${\mathcal C}$ and $\widetilde {\mathcal C}$ are nonempty, disjoint, and there are no edges between ${\mathcal C}$ and $\widetilde{{\mathcal C}}$. Suppose that every vertex of $T$ is adjacent to exactly $A$, $A<\delta$, vertices of $T$, and every vertex of $\overline{T}$ is adjacent to exactly $C$, $C\ge1$, vertices of ${\mathcal C}$ and to exactly $B$ vertices of $\widetilde{\mathcal C}$. Then the following holds.

1. We have ${\mathcal C}_1=\overline{T}$, ${\mathcal C}_2=\widetilde{\mathcal C}$, and ${\mathcal C}$ is a CRC with parameter matrix
\[
M=\begin{pmatrix}
A& \delta-A&0 \\
C& \delta-B-C &B\\
0&\delta-A&A
\end{pmatrix}.
\]
2. The code $T$ is a CRC in $\Gamma$ with parameter matrix \[
M=\begin{pmatrix}
A& \delta-A \\
B+C& \delta-B-C \\
\end{pmatrix}.
\]

\end{proposition}
\begin{proof} 1.
Since there are no edges between ${\mathcal C}$ and $\widetilde{{\mathcal C}}$, all $A$ neighbors of a vertex of ${\mathcal C}$ in $T$ are in ${\mathcal C}$, and the remaining $\delta-A\ge1$ are in $\overline{T}$; the same holds for $\widetilde {\mathcal C}$. In particular $\overline T\neq\varnothing$. Since $C\ge1$, every vertex of $\overline T$ is adjacent to ${\mathcal C}$, so $\overline{T}={\mathcal C}_1$. The vertices of $\widetilde{{\mathcal C}}$ are not adjacent to ${\mathcal C}$ and have neighbors in $\overline T={\mathcal C}_1$, hence $\widetilde{{\mathcal C}}={\mathcal C}_2$. The rows of the matrix are obtained by direct counting.

2. Follows directly from the parameters established in the previous item.

\end{proof}
\section{New constructions of CRCs over $\Z_6$ in the Lee metric}\label{S3}
Each construction of Sections~3.1--3.3 gives a code in $G_6^n$, to which K2 is then applied; therefore the resulting CRCs also exist in $G_{6r}^n$ for all $r\ge1$ and in $G^n$. The construction of Section~3.4 works for any even $q\geq 4$.
\subsection{A concatenation construction from ternary punctured perfect codes}
The general form of this construction resembles the Solov'eva concatenation construction for perfect and extended perfect codes \cite{Sol}. We say that a triple $(Q^0,Q^1,Q^2)$ of subsets of $\Z_6^m$ is \emph{admissible} if $Q^0,Q^1,Q^2$ are nonempty and pairwise disjoint, their union $S$ is an independent set of the graph $G_6^m$, and every vertex outside $S$ is adjacent to exactly one vertex of each of $Q^0,Q^1,Q^2$. This is equivalent to $S$ being a CRC with matrix $\begin{pmatrix}0&2m\\3&2m-3\end{pmatrix}$, each $Q^i$ being a CRC with matrix $\begin{pmatrix}0&2m&0\\1&2m-3&2\\0&2m&0\end{pmatrix}$, and $Q^i_2=S\setminus Q^i$. The same definition is used for subsets of $\Z_3^m$ and the graph $G_3^m=H(m,3)$.
\begin{lemma}\label{L:ternary}
For $l\ge2$ and $m=\frac{3^l-3}{2}$ there exists an admissible triple in $\Z_6^m$.
\end{lemma}
\begin{proof}
Let  $P$ be the ternary perfect (for example, Hamming code)  in $H(m+1,3)=G_3^{m+1}$.  Put $Q'^{i}=\{y\in\Z_3^m:\ (y,i)\in P\}$, $\alpha\in\Z_3$, and $S'=\bigcup_{i\in \Z_3} Q'^{i}$. Using code terms,  $S'$ is the punctured Hamming code,  and the sets $Q'^{i}$ are shortened codes.
 Thus $(Q'^0,Q'^1,Q'^2)$ is an admissible triple in $G_3^m$, because each shortened perfect code $Q'^{i}$, $i\in \Z_3$ is a CRC with required parameter matrix \cite{ShiWuKrotov2022}.

The triple $(Q^0,Q^1,Q^2)=(\phi^{-1}_3(Q'^0),\phi^{-1}_3(Q'^1),\phi^{-1}_3(Q'^2))$ under $\phi_3:\Z_6^m\to\Z_3^m$ is an admissible triple in $G_6^m$, since $\phi_3$ maps the neighbors of each vertex bijectively onto the neighbors of its image.
\end{proof}
\begin{theorem}\label{T}
Let $(Q^0,Q^1,Q^2)$ be an admissible triple in $\Z_6^m$, $R^i=\{2i,2i+1\}\subset\Z_6$, $i=0,1,2$. Then the code
\[
{\mathcal C}=\bigcup_{i\in \Z_3}Q^i\times R^i
\]
is a CRC in $G_6^{m+1}$ with parameter matrix
\[
\begin{pmatrix}
1& 2m+1& 0 \\
1& 2m & 1\\
0& 2m+1 & 1
\end{pmatrix}.
\]
\end{theorem}
\begin{proof}
Let $\widetilde{R^i}=\{2i+3,2i+4\}\subset\Z_6$ .
In the cycle $G_6$ each of the families $\{R^i:i\in \Z_3\}$ and $\{\widetilde{R^i}:i\in \Z_3\}$ partitions $\Z_6$, and for each $i$ the sets $R^i$, $\widetilde{R^i}$ and $N^i=\Z_6\setminus(R^i\cup \widetilde{R^i})$ have the following properties: a vertex of $R^i$ (of $\widetilde{R^i}$) has one neighbor in $R^i$ (in $\widetilde{R^i}$) and none in $\widetilde{R^i}$ (in $R^i$); a vertex of $N^i$ has one neighbor in $R^i$ and one in $\widetilde{R^i}$, see Fig.~\ref{fig:Bcycle} a, b. In other words, for any $i\in \Z_3$ $R^i$, $N^i$, $\widetilde{R^i}$ is the distance partition w.r.t. the CRC $R^i$ in $G_6$ with the parameter matrix \[
\begin{pmatrix}
1& 1& 0 \\
1& 0 & 1\\
0& 1 & 1
\end{pmatrix},
\]
see Fig.~\ref{fig:Bcycle} c.

We now check the conditions of Proposition~\ref{P1} for $T={\mathcal C}\cup{\widetilde{\mathcal C}}$, $\delta=2m+2$, where ${\widetilde{\mathcal C}}=\bigcup_{i\in \Z_3}Q^i\times \widetilde{R^i}$. Since the sets $Q^i$ are pairwise disjoint, for $x\in Q^i$ we have $(x,\gamma)\in T$ if and only if $\gamma\in R^i\cup\widetilde{R^i}$; in particular ${\mathcal C}\cap\widetilde{\mathcal C}=\varnothing$.

Consider a vertex $(x,\beta)\in{\mathcal C}$, $x\in Q^i$, $\beta\in R^i$. Its neighbors of the form $(x',\beta)$ are not in $T$, since $x'\notin S$ by the independence of the set $S$. By the remark at the beginning of the proof, among the neighbors $(x,\beta\pm1)$ exactly one lies in $Q^i\times R^i\subseteq{\mathcal C}$ and the other is in $Q^i\times N^i$, so not in $T$. Hence $(x,\beta)$ has exactly one neighbor in $T$, and it lies in ${\mathcal C}$. The same holds for $\widetilde{\mathcal C}$; in particular, there are no edges between ${\mathcal C}$ and $\widetilde{\mathcal C}$, and $A=1$.

Let $(x,\beta)\notin T$ and $x\in Q^i\subset S$ for some $i\in \Z_3$. Then $\beta$ belongs to $N^i=\Z_6\setminus(R^i\cup \widetilde{R^i})$. The neighbors of the vertex $(x,\beta)$ of the form $(x',\beta)$ do not lie in $T$ (again by the independence of $S$), and of the neighbors $(x,\beta\pm1)$ one lies in $Q^i\times R^i\subset {\mathcal C}$ and one in $Q^i\times \widetilde{R^i}\subset \widetilde{\mathcal C}$ by the remark at the beginning of the proof.

Let $(x,\beta)\notin T$ and $x\notin S$. Then the neighbors $(x,\beta\pm1)$ are not in $T$ because $x$ is not in $S=\bigcup_{i\in \Z_3} Q^i$. Moreover, by definition of admissible triple $(Q^0,Q^1,Q^2)$, a vertex $x$, $x\notin S$, has exactly one neighbor $x^j\in Q^j$ for each $j\in \Z_3$, these are all its neighbors in $S$, and the vertex $(x^j,\beta)$ is in $Q^j\times R^j\subset {\mathcal C}$ (in $Q^j\times \widetilde{R^j}\subset \widetilde{\mathcal C}$) if and only if $\beta\in R^j$ ($\beta\in \widetilde{ R^j}$), which holds for exactly one $j$. In both cases $C=B=1$, and the statement follows from Proposition~\ref{P1}.
\end{proof}

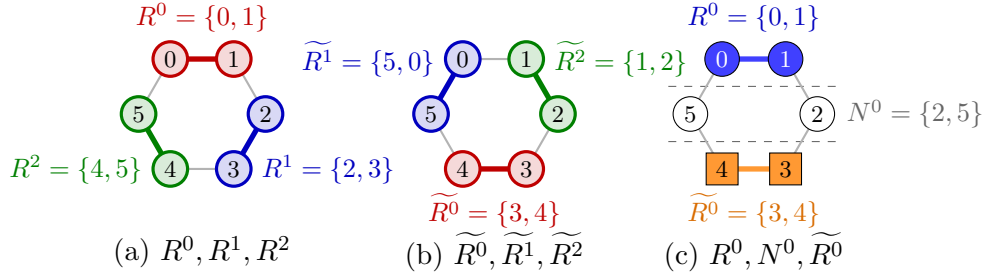
\begin{figure}[ht]
\centering
\begin{tikzpicture}[scale=0.73,
  vtx/.style={circle,draw=black,fill=white,inner sep=1pt,minimum size=13pt,font=\footnotesize},
  lbl/.style={font=\small}]
\colorlet{c0}{red!75!black}
\colorlet{c1}{blue!75!black}
\colorlet{c2}{green!50!black}
\begin{scope}
  \foreach \k in {0,...,5}{\coordinate (a\k) at ({120-60*\k}:1.15);}
  \draw[gray!60,thick] (a0)--(a1)--(a2)--(a3)--(a4)--(a5)--cycle;
  \foreach \i/\x/\y in {0/0/1,1/2/3,2/4/5}{
    \draw[c\i,line width=2pt] (a\x)--(a\y);
  }
  \node[lbl,c0] at (0,1.75) {$R^{0}=\{0,1\}$};
  \node[lbl,c1,anchor=west] at (0.9,-1.0) {$R^{1}=\{2,3\}$};
  \node[lbl,c2,anchor=east] at (-0.9,-1.0) {$R^{2}=\{4,5\}$};
  \foreach \k/\i in {0/0,1/0,2/1,3/1,4/2,5/2}{
    \node[vtx,draw=c\i,fill=c\i!15,very thick] at (a\k) {$\k$};
  }
  \node at (0,-2.5) {(a) $R^0,R^1,R^2$};
\end{scope}
\begin{scope}[xshift=5.3cm]
  \foreach \k in {0,...,5}{\coordinate (b\k) at ({120-60*\k}:1.15);}
  \draw[gray!60,thick] (b0)--(b1)--(b2)--(b3)--(b4)--(b5)--cycle;
  \foreach \i/\x/\y in {0/3/4,1/5/0,2/1/2}{
    \draw[c\i,line width=2pt] (b\x)--(b\y);
  }
  \node[lbl,c0] at (0,-1.75) {$\widetilde{R^{0}}=\{3,4\}$};
  \node[lbl,c1,anchor=east] at (-0.9,1.0) {$\widetilde{R^{1}}=\{5,0\}$};
  \node[lbl,c2,anchor=west] at (0.9,1.0) {$\widetilde{R^{2}}=\{1,2\}$};
  \foreach \k/\i in {0/1,1/2,2/2,3/0,4/0,5/1}{
    \node[vtx,draw=c\i,fill=c\i!15,very thick] at (b\k) {$\k$};
  }
  \node at (0,-2.5) {(b) $\widetilde{R^0},\widetilde{R^1},\widetilde{R^2}$};
\end{scope}
\begin{scope}[xshift=10.0cm]
  \colorlet{cN}{black!60}
  \foreach \k in {0,...,5}{\coordinate (d\k) at ({120-60*\k}:1.15);}
  \draw[cN,dashed] (-1.55,0.5)--(1.55,0.5);
  \draw[cN,dashed] (-1.55,-0.5)--(1.55,-0.5);
  \draw[gray!60,thick] (d0)--(d1)--(d2)--(d3)--(d4)--(d5)--cycle;
  \draw[blue!75,line width=2pt] (d0)--(d1);
  \draw[orange!80,line width=2pt] (d3)--(d4);
  \foreach \k in {0,1}{\node[vtx,draw=black,fill=blue!75,text=white] at (d\k) {$\k$};}
  \foreach \k in {2,5}{\node[vtx,draw=black,fill=white] at (d\k) {$\k$};}
  \foreach \k in {3,4}{\node[vtx,rectangle,draw=black,fill=orange!80,minimum size=12pt] at (d\k) {$\k$};}
  \node[lbl,blue!75!black] at (0,1.75) {$R^0=\{0,1\}$};
  \node[lbl,cN,anchor=west] at (1.5,0) {$N^0=\{2,5\}$};
  \node[lbl,orange!85!black] at (0,-1.75) {$\widetilde{R^0}=\{3,4\}$};
  \node at (0,-2.5) {(c) $R^0,N^0,\widetilde{R^0}$};
\end{scope}
\end{tikzpicture}
\caption{Two partitions of the graph $G_6$  (a) $R^0\cup R^1\cup R^2$, $R^i=\{2i,2i+1\}$ and (b) $\widetilde{R^0}\cup \widetilde{R^1}\cup \widetilde{R^2}$, $\widetilde{R^i}=R^i+3$; (c) the distance partition $(R^0,N^0,\widetilde{R^0})$ of $G_6$ with respect to the CRC $R^0$ with parameter matrix $\left(\begin{smallmatrix}1&1&0\\1&0&1\\0&1&1\end{smallmatrix}\right)$.}
\label{fig:Bcycle}
\end{figure}

\begin{corollary}\label{Cor:cascade}
For all $l,r\ge1$, CRCs with parameter matrix $\begin{pmatrix}
1& 3^l-2& 0 \\
1& 3^l-3 & 1\\
0& 3^l-2 & 1
\end{pmatrix}$ exist in the graphs $G_{6r}^{\frac{3^l-1}{2}}$ and $G^{\frac{3^l-1}{2}}$.
\end{corollary}
\begin{proof}
For $l=1$ the required code in $G_6$ is the edge $R^0$, see Fig. \ref{fig:Bcycle}.c. For $l\ge2$ we apply Theorem~\ref{T} to the admissible triple from Lemma~\ref{L:ternary}: $m+1=\frac{3^l-1}{2}$, $2m+1=3^l-2$. We use K2 to obtain CRCs over  $\Z_{6r}$ and $\Z$.
\end{proof}
\subsection{CRCs in the $l_\infty$ metric}
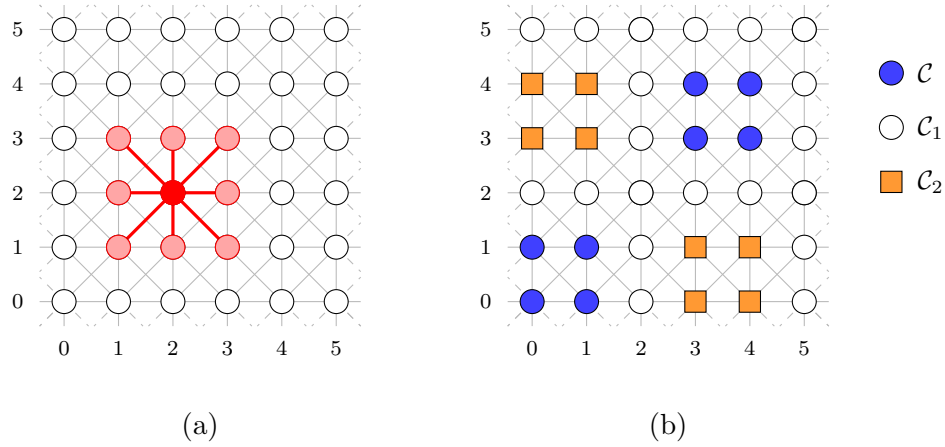
\begin{figure}[ht]
\centering
\begin{tikzpicture}[scale=0.72,
  vtx/.style={circle,draw=black,fill=white,inner sep=0pt,minimum size=9pt},
  cvtx/.style={circle,draw=black,inner sep=0pt,minimum size=9pt}]
\begin{scope}
  \foreach \i in {0,...,5}{
    \draw[gray!55] (-0.45,\i) -- (5.45,\i);
    \draw[gray!55] (\i,-0.45) -- (\i,5.45);
  }
  \foreach \i in {0,...,4}{\foreach \j in {0,...,4}{
    \draw[gray!55] (\i,\j) -- (\i+1,\j+1);
    \draw[gray!55] (\i+1,\j) -- (\i,\j+1);
  }}
  \foreach \i in {0,...,5}{
    \foreach \d in {-1,1}{
      \draw[gray!55,dashed] (0,\i) -- (-0.45,\i+0.45*\d);
      \draw[gray!55,dashed] (5,\i) -- (5.45,\i+0.45*\d);
      \draw[gray!55,dashed] (\i,0) -- (\i+0.45*\d,-0.45);
      \draw[gray!55,dashed] (\i,5) -- (\i+0.45*\d,5.45);
    }
  }
  \foreach \dx in {-1,0,1}{\foreach \dy in {-1,0,1}{
    \draw[red,very thick] (2,2) -- (2+\dx,2+\dy);
  }}
  \foreach \i in {0,...,5}{\foreach \j in {0,...,5}{
    \node[vtx] at (\i,\j) {};
  }}
  \foreach \dx in {-1,0,1}{\foreach \dy in {-1,0,1}{
    \node[vtx,fill=red!35,draw=red] at (2+\dx,2+\dy) {};
  }}
  \node[vtx,fill=red,draw=red,minimum size=9pt] at (2,2) {};
  \foreach \i in {0,...,5}{
    \node[font=\scriptsize] at (\i,-0.85) {$\i$};
    \node[font=\scriptsize] at (-0.85,\i) {$\i$};
  }
  \node at (2.5,-2.3) {(a) };
\end{scope}
\begin{scope}[xshift=8.6cm]
  \foreach \i in {0,...,5}{
    \draw[gray!55] (-0.45,\i) -- (5.45,\i);
    \draw[gray!55] (\i,-0.45) -- (\i,5.45);
  }
  \foreach \i in {0,...,4}{\foreach \j in {0,...,4}{
    \draw[gray!55] (\i,\j) -- (\i+1,\j+1);
    \draw[gray!55] (\i+1,\j) -- (\i,\j+1);
  }}
  \foreach \i in {0,...,5}{
    \foreach \d in {-1,1}{
      \draw[gray!55,dashed] (0,\i) -- (-0.45,\i+0.45*\d);
      \draw[gray!55,dashed] (5,\i) -- (5.45,\i+0.45*\d);
      \draw[gray!55,dashed] (\i,0) -- (\i+0.45*\d,-0.45);
      \draw[gray!55,dashed] (\i,5) -- (\i+0.45*\d,5.45);
    }
  }
  \foreach \i in {0,...,5}{\foreach \j in {2,5}{
    \node[cvtx,fill=white] at (\i,\j) {};
    \node[cvtx,fill=white] at (\j,\i) {};
  }}
  \foreach \i in {0,1}{\foreach \j in {0,1}{
    \node[cvtx,fill=blue!75] at (\i,\j) {};
    \node[cvtx,fill=blue!75] at (\i+3,\j+3) {};
  }}
  \foreach \i in {0,1}{\foreach \j in {3,4}{
    \node[cvtx,fill=orange!80,rectangle,minimum size=8pt] at (\i,\j) {};
    \node[cvtx,fill=orange!80,rectangle,minimum size=8pt] at (\j,\i) {};
  }}
  \foreach \i in {0,...,5}{
    \node[font=\scriptsize] at (\i,-0.85) {$\i$};
    \node[font=\scriptsize] at (-0.85,\i) {$\i$};
  }
  \node at (2.5,-2.3) {(b) };
  \begin{scope}[xshift=6.6cm,yshift=1.2cm]
    \node[cvtx,fill=blue!75] at (0,3) {};
    \node[anchor=west,font=\small] at (0.3,3) {${\mathcal C}$ };
    \node[cvtx,fill=white] at (0,2) {};
    \node[anchor=west,font=\small] at (0.3,2) {${\mathcal C}_1$ };
    \node[cvtx,fill=orange!80,rectangle,minimum size=8pt] at (0,1) {};
    \node[anchor=west,font=\small] at (0.3,1) {${\mathcal C}_2$ };
  \end{scope}
\end{scope}
\end{tikzpicture}
\caption{(а) The adjacency in the graph $D_6^2$   (b) The CRC ${\mathcal C}=\{0,1\}^2\cup\{3,4\}^2$ in $D_6^2$  with parameter matrix $\left(\begin{smallmatrix}3&5&0\\2&4&2\\0&5&3\end{smallmatrix}\right)$.}
\label{fig:D62}
\end{figure}

By K4, every CRC in $D_6^l$ with $l_{\infty}$ metric provides a CRC with the same parameter matrix in $G_6^{\frac{3^l-1}{2}}$.
Consider the code $T=\{0,1,3,4\}^l\subset\Z_6^l$. We call the symbols $0,1$ \emph{lower} and $3,4$ \emph{upper}. Let ${\mathcal C}$ ($\widetilde{\mathcal C}$) be the set of words of $T$ with an even (odd) number of lower symbols. A neighbor $x+\varepsilon$, $\varepsilon\in\{0,\pm1\}^l\setminus\{0^l\}$, of a vertex $x$ in $D_6^l$ is obtained by changing each symbol by $0$ or $\pm1$.
\begin{theorem}\label{T2}
The code ${\mathcal C}$ is a CRC in $D_6^{l}$ with parameter matrix
\[
\begin{pmatrix}
2^l-1& 3^l-2^l &0\\
2^{l-1}& 3^l-1-2^l &2^{l-1}\\
0& 3^l-2^l &2^{l}-1
\end{pmatrix}.
\]
\end{theorem}
\begin{proof}
Let $x$ be in $T$. A symbol $x_k\in\{0,1,3,4\}$ can be changed by $0$ or $\pm1$ while staying in $\{0,1,3,4\}$ in exactly two ways: by leaving it unchanged or by replacing it with its ``partner'' in this set (that is, $0\leftrightarrow1$ or $3\leftrightarrow4$); in this case a lower symbol remains lower and an upper symbol remains upper. Therefore $x$ has exactly $2^l-1$ neighbors in $T$, and all of them lie in the same one of the codes ${\mathcal C}$, $\widetilde{\mathcal C}$ as $x$; in particular, there are no edges between ${\mathcal C}$ and $\widetilde{\mathcal C}$.

Let $x\notin T$, i.e. $w\ge1$ symbols of $x$ are equal to $2$ or $5$. The symbol $2$ can be replaced by $1$ (lower) or $3$ (upper), the symbol $5$ by $0$ (lower) or $4$ (upper), and each of the lower and upper symbols in two ways preserving its type. Since $w\ge1$, all $2^l$ resulting words are different from $x$, i.e. they are neighbors of $x$ in $T$, and, since the type of a symbol equal to $2$ or $5$ can be chosen freely, exactly half of them have an even number of lower symbols. Thus $C=B=2^{l-1}$, $A=2^l-1<3^l-1$, and the statement follows from Proposition~\ref{P1}.
\end{proof}

\begin{corollary}\label{Cor:linf}
For any $l\geq 1$ and $r\geq 1$ there exists a CRC in $G_{6r}^{\frac{3^l-1}{2}}$ and in $G^{\frac{3^l-1}{2}}$ with the parameter matrix of Theorem~\ref{T2}.
\end{corollary}
\begin{proof}
By Theorem~\ref{T2} and construction K4 (with $q=6$, $m=l$), the preimage of ${\mathcal C}$ is a CRC in $G_6^{\frac{3^l-1}{2}}$ with the same parameter matrix; then construction K2 from Section \ref{SS:cover} gives the CRCs in $G_{6r}^{\frac{3^l-1}{2}}$ and $G^{\frac{3^l-1}{2}}$.
\end{proof}
\begin{remark}
The example of the code from Theorem~\ref{T2} for $l=2$ in Fig.~\ref{fig:D62} was proposed by S.V.~Avgustinovich.
\end{remark}
{\bf Open problem (S.V. Avgustinovich).} Obtain a classification of CRCs in the graphs $D_q^n$ (Cayley graphs of the group $\Z_q^n$ with the $l_{\infty}$ metric), $n\geq 2$.
\subsection{CRCs from binary perfect codes}

Let $n=2^l$, $l\ge1$, and let $P$ be a binary perfect code of length $n-1$.

By K1, the code ${\mathcal B}=\phi_2^{-1}(P)\subset\Z_6^{n-1}$ is a CRC in $G_6^{n-1}$ with matrix $\begin{pmatrix}0&2n-2\\2&2n-4\end{pmatrix}$.

From the parameter matrix of the perfect code $P$ in $H(n-1,2)$ we have the following obvious fact:

$$\textbf{(H)} \mbox{ if } z\notin {\mathcal B}, \mbox{ then the neighbors of }z\mbox{ in }{\mathcal B}\mbox{ are exactly }z+e^j\mbox{ and }z-e^j\mbox{ for a unique }j.$$

We write the vertices of the graph $G_6^n$ in the form $x=(z,\alpha)$, $z\in\Z_6^{n-1}$, $\alpha\in\Z_6$, and consider the homomorphism $\psi$ from $\Z_6^n$ onto $\Z_3$:
\[ \psi(z,\alpha)=\alpha-\sum_{i=1}^{n-1} z_i \bmod 3 .
\]

Note that the homomorphism $\psi$ maps the graph $G_6^n$ onto a multigraph which is the complete graph $K_3$ with each edge of multiplicity $n$.

Put
\[
T=\ker\psi=\bigcup_{z\in \Z_6^{n-1}} z\times s(z),\qquad s(z)=\Bigl\{\sum z_i,\ \sum z_i+3\Bigr\}\subset\Z_6,
\]
\[
{\mathcal C}=\bigcup_{z\in {\mathcal B}} z\times s(z),\qquad \widetilde{\mathcal C}=\bigcup_{z\notin {\mathcal B}} z\times s(z).
\]

\begin{theorem}\label{T3}
For each $l\geq 1$ the code ${\mathcal C}$ is a CRC in $G_6^{2^l}$ with parameter matrix
$\begin{pmatrix}
0& 2^{l+1}&0 \\
1& 2^l &2^l-1\\
0&2^{l+1}&0
\end{pmatrix}$.
Moreover, for all $l\ge1$, $r\ge1$ CRCs with the same parameter matrix exist in $G_{6r}^{2^l}$ and in $G^{2^l}$.
\end{theorem}
\begin{proof}
Since $\psi(T)$ is a single vertex of the multigraph $K_3$ (with each edge of multiplicity $n$), which has no loops, and since $\psi$ is a homomorphism from $G_{6}^n$ onto this multigraph, no pair of vertices of $T={\mathcal C}\cup \widetilde{\mathcal C}$ is adjacent, i.e. $A=0$; in particular, there are no edges between ${\mathcal C}$ and ${\widetilde{\mathcal C}}$.

 Let $x=(z,\alpha)\notin T$ and $\beta=\psi(x)\in\Z_3\setminus\{0\}=\{1,-1\}$. Denote
\[
y^n=(z,\ \alpha-\beta),\qquad y^k=x+\beta e^k=(z+\beta e^k,\ \alpha),\, k<n.
\]
Here $e^k=(0,\ldots,0,1,0,\ldots,0)$ is the word of length $n-1=2^l-1$ with the one in the $k$th position.

For each $1\le k\le n$, exactly one of the neighbors $x\pm e^k$ lies in $T$, namely $y^k$, since $\psi(y^n)=\beta-\beta=0$ and $\psi(y^k)=\beta-\beta=0$ for $k<n$, while the other neighbors have $\psi(z,\alpha+\beta)=2\beta\ne0$ and $\psi(z-\beta e^k,\alpha)=2\beta\ne0$. Thus $x$ has exactly $n$ neighbors $y^1,\ldots,y^n$ in $T$ and exactly $2n-n=n$ neighbors in $\overline T$. By the definition of $\mathcal C$ we have
\[
y^n\in {\mathcal C}\iff z\in {\mathcal B},\qquad y^k\in {\mathcal C}\iff z+\beta e^k\in {\mathcal B}\mbox{ for all }k<n.
\]
Recall that the code ${\mathcal B}$ is an independent set in $G_6^{n-1}$, since the parameter matrix of this CRC has $0$ in the upper left corner. We have the following cases.

 Case A. Let $z\in {\mathcal B}$. Then $y^n\in {\mathcal C}$, and by the independence of ${\mathcal B}$, $z+\beta e^k\notin {\mathcal B}$ for all $k<n$. The vertex $x$ has exactly $n-1$ neighbors $y^1,\ldots,y^{n-1}$ in $\widetilde{\mathcal C}$ and exactly one, namely $y^n$, in $\mathcal C$.

Case B. Let $z\notin {\mathcal B}$. Then $y^n\in\widetilde{\mathcal C}$, and by (H) the neighbors of $z$ in ${\mathcal B}$ are exactly $z+e^j$ and $z-e^j$ for a unique $j<n$; hence $z+\beta e^k\in {\mathcal B}\iff k=j$. Therefore $y^j\in {\mathcal C}$ and $y^k\in\widetilde{\mathcal C}$ for $k\ne j$.
\noindent In both cases $x$
has exactly $C=1$ neighbors in ${\mathcal C}$ and $B=n-1$ in $\widetilde{\mathcal C}$.
 The conditions of Proposition~\ref{P1} are satisfied with $\delta=2n$, $A=0$, $C=1\ge1$, $B=n-1$, and the parameter matrix of the code ${\mathcal C}$ is
\[
M=\begin{pmatrix}
A& \delta-A&0 \\
C& \delta-C-B &B\\
0&\delta-A&A
\end{pmatrix}
=
\begin{pmatrix}
0& 2n&0 \\
1& 2n-(n-1)-1 &n-1\\
0&2n&0
\end{pmatrix}
=
\begin{pmatrix}
0& 2n&0 \\
1& n &n-1\\
0&2n&0
\end{pmatrix}.
\]
We use K2 to obtain a CRC in $G_{6r}^n$ and $G^n$.
\end{proof}
\noindent\emph{Parity-check matrix.} Let $P$ be the binary Hamming code with an $l\times(n-1)$ parity-check matrix $H$. The code ${\mathcal B}=\{z\in\Z_6^{n-1}:\ Hz^\top\equiv 0\pmod 2\}$ has the parity-check matrix $3H$ over $\Z_6$, and $T=\ker\psi=\{(z,\alpha):\ 2(\alpha-\sum z_i)=0\}$. Since ${\mathcal C}=({\mathcal B}\times\Z_6)\cap T$, the code ${\mathcal C}$ is linear with parity-check matrix
\[
\begin{pmatrix}3H&0\\ 4\cdot\mathbf 1&2\end{pmatrix}\quad\text{over }\Z_6 ,
\] where $\mathbf 1$ is the all-ones row.
\subsection{``Lengthening'' of a doubled perfect code}
Let $q\ge4$ be even and let $P$ be a perfect code in $G_q^m$. We write the vertices of $G_q^{2m+1}$ in the form $(u,v,\alpha)$, $u,v\in\Z_q^m$, $\alpha\in\Z_q$, and put $x=u+v$; for $x\in\Z_q^m$ let $t(x)=\{\sum\limits_{i=1}^{m} x_i+2j:\ 0\le j<q/2\}\subset\Z_q$. Consider the codes
\begin{equation}\label{2x}
{\mathcal C}=\bigcup_{y\in \Z_q^m,\ x\in P}y \times (x-y)\times t(x),\qquad
\widetilde{\mathcal C}=\bigcup_{y\in \Z_q^m,\ x\notin P}y \times (x-y)\times t(x),\qquad T={\mathcal C}\cup \widetilde{\mathcal C}.
\end{equation}
Thus $T=\{(u,v,\alpha):\ \alpha+{\sum\limits_{i=1}^{m}} (u_i+ v_i)\ \text{is even}\}$ is one of the parts of the bipartite graph $G_q^{2m+1}$, $\overline T$ is the other part, and ${\mathcal C}$ and $\widetilde{\mathcal C}$ differ by the condition $x=u+v\in P$ or $x\notin P$.
\begin{theorem}\label{T4}
If for even $q\ge4$ there exists a perfect code $P$ in $G_q^m$, then the code ${\mathcal C}$ in \eqref{2x} is a CRC in $G_q^{2m+1}$ with parameter matrix $\begin{pmatrix}
0& 4m+2&0 \\
2& 0 &4m\\
0&4m+2&0
\end{pmatrix}$.
\end{theorem}
\begin{proof}
The part $T$ is independent, so $A=0$ and there are no edges between ${\mathcal C}$ and $\widetilde{\mathcal C}$. Let $z=(u,v,\alpha)$ be a vertex in the part $\overline{T}$, $x=u+v$. All neighbors of $z$ lie in the part $T$; the neighbors $(u\pm e^i,v,\alpha)$ and $(u,v\pm e^i,\alpha)$, $i=1,\ldots,m$, correspond to the word $x\pm e^i$, and the neighbors $(u,v,\alpha\pm1)$ correspond to the word $x$.
If $x\in P$, then $x\pm e^i\notin P$ by the perfectness of the code $P$, and the neighbors of $z$ in ${\mathcal C}$ are exactly $(u,v,\alpha\pm1)$. If $x\notin P$, then $x$ has a unique neighbor $x+\beta e^i\in P$ (where $\beta\in\{1,-1\}$), and the neighbors of $z$ in ${\mathcal C}$ are exactly $(u+\beta e^i,v,\alpha)$ and $(u,v+\beta e^i,\alpha)$. In both cases $C=2$, $B=4m$, and the statement follows from Proposition~\ref{P1}.
\end{proof}
\noindent\emph{Parity-check matrix.} If $P$ is a linear code with parity-check matrix $H$ over $\Z_q$, then ${\mathcal C}=\{(u,v,\alpha):\ H(u+v)^\top=0,\ \alpha+\sum\limits_{i=1}^{m} (u_i+ v_i)\ \text{is even}\}$ is linear with parity-check matrix
\[
\begin{pmatrix}H&H&0\\ \frac q2\mathbf 1&\frac q2\mathbf 1&\frac q2\end{pmatrix}.
\]
\begin{remark}
If $(2m+1)\mid q$, then a perfect code in $G_q^m$ exists (Section~\ref{SS:simple}). According to Horak, AlBdaiwi and Milazzo \cite{AHM}, perfect codes in $G_q^m$ exist for all $m$ such that the product of all prime divisors of $2m+1$ divides $q$. In particular, such a code exists in $G_6^4$ (for example, the preimage under $\phi_3$ of the ternary Hamming code of length $4$), so Theorem~\ref{T4} gives a CRC in $G_6^9$. 
\end{remark}
\section{CRCs in the lattice graphs $G_q^n$, $n\leq 4$, $q=5$, $6$, $7$}\label{S4}

The enumeration of the parameter matrices of all CRCs in the graphs $G_q^n$ for $q\leq 7$, $n\leq 4$ was carried out by solving an integer linear programming problem. The program was written in Python, and the problems were solved with the HiGHS solver.

Let ${\mathcal C}$ be a CRC with covering radius $\rho$ in a graph with vertex set $V$ and adjacency matrix $L$. Its \emph{characteristic matrix} is the $|V|\times(\rho+1)$ binary matrix $X$ with $X_{u,i}=1$ if $u\in{\mathcal C}_i$ and $X_{u,i}=0$ otherwise. It is easy to see \cite[Proposition 1.4]{CRC} that a binary matrix $X$ is the characteristic matrix of a CRC with a tridiagonal parameter matrix $M$ (with nonzero sub- and superdiagonal entries) if and only if
\[
LX=XM\qquad\text{and}\qquad \sum_{i=0}^{\rho} X_{u,i}=1\quad\text{for every vertex } u .
\]
Thus a CRC with parameter matrix $M$ exists if and only if this system has a binary solution.

Note that the case of the graph $G_7^4$ with $7^4=2401$ vertices is already borderline, since the number of variables is already very large. However, the graph has small degree and fairly large diameter, which has a positive effect on the speed of the solver. For instance, already in the Hamming graph $H(4,6)$ with $6^4=1296$ vertices and degree $20$, the solver cannot cope with the problem of the existence of CRCs.

Before solving the problem for a candidate matrix $M$ with covering radius $\rho$, we checked that $M$ satisfies the following necessary conditions.

1. \emph{Monotonicity of the off-diagonal entries.} As proved in \cite{AvgVas22}, the parameter matrix of every CRC in the lattice graphs $G_q^n$ and $G^n$ satisfies $M_{i,i+1}\geq M_{i+1,i+2}$ for $0\leq i\leq \rho-2$ and $M_{i,i-1}\leq M_{i+1,i}$ for $1\leq i\leq \rho-1$, i.e. $b_0\ge b_1\ge\ldots\ge b_{\rho-1}$ and $c_1\le c_2\le\ldots\le c_\rho$.

2. \emph{Integrality of the cell sizes.} The sizes of the cells ${\mathcal C}_i$ of the distance partition of a graph of order $q^n$ are determined by the relations \cite[Lemma 7.2]{CRC}
\[
M_{i,i+1}|{\mathcal C}_i|=M_{i+1,i}|{\mathcal C}_{i+1}|\quad(0\leq i\leq\rho-1),\qquad \sum_{i=0}^{\rho} |{\mathcal C}_i|=q^n,
\]
and they must be integers.

3. \emph{Integrality of eigenvalues.} By Lloyd theorem \cite[Lemma 7.4]{CRC}, every eigenvalue of the parameter matrix of a CRC is an eigenvalue of the graph. The eigenvalues of the cycle $G_q$ are $2\cos(2\pi k/q)$, $k=0,\ldots,q-1$, and the eigenvalues of $G_q^n$ are all sums of $n$ eigenvalues of $G_q$. For example, the eigenvalues of $G_6$ are $2,1,-1,-2$, and for $n\geq 2$ the eigenvalues of $G_6^n$ are all integers from $-2n$ to $2n$; hence for $q=6$ all eigenvalues of $M$ must be integers.

These conditions significantly reduce the set of possible parameter matrices. For $G_6^4$, up to passing to the opposite code, there are $496$ tridiagonal $3\times 3$ matrices with row sum $8$ and positive off-diagonal entries fulfilling the monotonicity condition, $76$ of them have integer cells, and only $41$ fulfill all three conditions.

In all tables of this section, CRCs are given by their intersection arrays $[b_0,\ldots,b_{\rho-1};\,c_1,\ldots,c_\rho]$ (see Section~\ref{S2}), which uniquely determine the parameter matrices. If ${\mathcal C}$ is a CRC with array $[b_0,\ldots,b_{\rho-1};\,c_1,\ldots,c_\rho]$, then the \emph{opposite code} ${\mathcal C}_\rho$ is a CRC with array $[c_\rho,\ldots,c_1;\,b_{\rho-1},\ldots,b_0]$; of each such pair of arrays, only one is listed.
Notation in the column ``construction'': \textbf{pc} is a perfect code; $k$\,\textbf{pc} is a $k$-fold perfect code obtained as the union of $k$ cosets of a linear perfect code (in a $\delta$-regular graph it is a CRC with array $[\delta+1-k;\,k]$; $1$\,pc is pc itself); \textbf{sing} is a one-vertex code; \textbf{even} is the even-weight code; \textbf{ext pc} is an extended perfect code; \textbf{even pc} is the even-weight subcode of a perfect code; \textbf{rep} is the ternary repetition code of length $3$, a well-known CRC in $H(3,3)$ with intersection array $[6,2;\,1,6]$, see Fig.~\ref{fig:H33}.
\begin{figure}[ht]
\centering
\begin{tikzpicture}[scale=0.9,
  vtx/.style={circle,draw=black,fill=white,inner sep=0pt,minimum size=6pt},
  cvtx/.style={circle,draw=black,inner sep=0pt,minimum size=9pt}]
\def\hthreecoords{%
  \foreach \i in {0,1,2}{\foreach \j in {0,1,2}{\foreach \k in {0,1,2}{
    \coordinate (v\i\j\k) at ({1.55*\i+0.7*\j},{1.55*\k+0.4*\j});
  }}}%
}
\def\hthreeedges{%
  \foreach \u in {0,1,2}{\foreach \v in {0,1,2}{
    \draw[gray!60] (v0\u\v) -- (v1\u\v) -- (v2\u\v);
    \draw[gray!60] (v0\u\v) to[bend left=25] (v2\u\v);
    \draw[gray!60] (v\u0\v) -- (v\u1\v) -- (v\u2\v);
    \draw[gray!60] (v\u0\v) to[bend right=25] (v\u2\v);
    \draw[gray!60] (v\u\v0) -- (v\u\v1) -- (v\u\v2);
    \draw[gray!60] (v\u\v0) to[bend right=25] (v\u\v2);
  }}%
}
\begin{scope}
  \hthreecoords
  \hthreeedges
  \foreach \i/\j/\k in {0/1/2,0/2/1,1/0/2,1/2/0,2/0/1,2/1/0}{
    \node[cvtx,fill=orange!80,rectangle,minimum size=8pt] at (v\i\j\k) {};
  }
  \foreach \i in {0,1,2}{\foreach \j in {0,1,2}{\foreach \k in {0,1,2}{
    \pgfmathtruncatemacro{\eqs}{(\i==\j)+(\j==\k)+(\i==\k)}
    \ifnum\eqs=1
      \node[cvtx,fill=white] at (v\i\j\k) {};
    \fi
  }}}
  \foreach \i in {0,1,2}{
    \node[cvtx,fill=blue!75] at (v\i\i\i) {};
    \node[font=\scriptsize,below right=1pt] at (v\i\i\i) {$\i\i\i$};
  }
  \node[cvtx,fill=blue!75] at (5.9,3.6) {};
  \node[anchor=west,font=\small] at (6.2,3.6) {${\mathcal C}$ };
  \node[cvtx,fill=white] at (5.9,2.8) {};
  \node[anchor=west,font=\small] at (6.2,2.8) {${\mathcal C}_1$ };
  \node[cvtx,fill=orange!80,rectangle,minimum size=8pt] at (5.9,2.0) {};
  \node[anchor=west,font=\small] at (6.2,2.0) {${\mathcal C}_2$ };
\end{scope}
\end{tikzpicture}
\caption{Hamming graph $H(3,3)$  and ternary code ${\mathcal C}=\{000,111,222\}$, which is CRC with parameter matrix $\begin{pmatrix}0&6&0\\1&3&2\\0&6&0\end{pmatrix}$.} 
\label{fig:H33}
\end{figure}
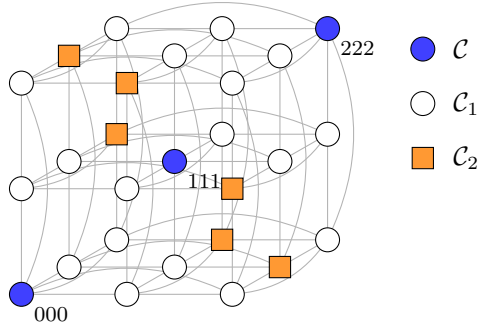

A graph written after a code name, as in ``sing $H(2,2)$'', means that the code is constructed in that graph and lifted to $G_q^n$ by one of the coverings K1--K3 of Section~\ref{SS:cover}; if no graph is given, the code is constructed in $G_q^n$ itself. The entries $G_q^m+2k$ and $G_q^m*k$ denote addition and multiplication (Section~\ref{SS:simple}) applied to a CRC from the table for $G_q^m$.
For codes obtained by the constructions of Section~\ref{S3}, the corresponding theorem or corollary is given together with its parameters. New constructions are marked with an asterisk.

Note that all CRCs found in $G_q^n$ for $q=5,6,7$ and $n\geq 3$ have small covering radii, namely at most $n$. For every $n\geq 1$, a CRC with covering radius $\lfloor q/2 \rfloor$ exists in $G_q^n$: take a one-vertex code in $G_q$ and add $n-1$ inessential coordinates (addition $+2(n-1)$). By \cite{AvgVas22}, the covering radius of a CRC in $G_q^{n}$ is at most $2n$, and this bound is attained, for example, when $q$ is divisible by $4$.

{\bf Open problem.} Is it true that for $q$ not divisible by $4$ and $n\geq \lfloor q/2\rfloor+1$, the covering radius of every CRC in $G_q^n$ is at most $n$?
\subsection{$G_q^n$, $q\le4$}
Since $G_2^n=H(n,2)$, $G_3^n=H(n,3)$ and $G_4^n\cong H(2n,2)$, these cases reduce to Hamming graphs; see the survey \cite[Section 7]{CRC}.

\subsection{$G_6^n$, $n\leq 4$}
\paragraph{$G_6$}
\begin{center}\begin{fitbox}
\begin{tabular}[t]{@{}l@{\quad}l@{}}
\toprule
\multicolumn{2}{@{}l}{$\rho=1$}\\
array & construction\\
\midrule
$[2;1]$ & pc $G_6$\\
$[2;2]$ & even\\
\bottomrule
\end{tabular}
\qquad
\begin{tabular}[t]{@{}l@{\quad}l@{}}
\toprule
\multicolumn{2}{@{}l}{$\rho=2$}\\
array & construction\\
\midrule
$[1,1;1,1]$ & Corollary~\ref{Cor:cascade}, $l=1$\\
\bottomrule
\end{tabular}
\qquad
\begin{tabular}[t]{@{}l@{\quad}l@{}}
\toprule
\multicolumn{2}{@{}l}{$\rho=3$}\\
array & construction\\
\midrule
$[2,1,1;1,1,2]$ & sing $G_6$\\
\bottomrule
\end{tabular}
\end{fitbox}\end{center}
\paragraph{$G_6^2$}
\begin{center}\begin{fitbox}
\begin{tabular}[t]{@{}l@{\quad}l@{}}
\toprule
\multicolumn{2}{@{}l}{$\rho=1$}\\
array & construction\\
\midrule
$[2;1]$ & $G_6+2$\\
$[2;2]$ & $G_6+2$\\
$[4;2]$ & $G_6*2$\\
$[4;4]$ & even\\
\bottomrule
\end{tabular}
\qquad
\begin{tabular}[t]{@{}l@{\quad}l@{}}
\toprule
\multicolumn{2}{@{}l}{$\rho=2$}\\
array & construction\\
\midrule
$[1,1;1,1]$ & $G_6+2$\\
$[4,2;1,2]$ & sing $H(2,3)$\\
$[2,2;2,2]$ & $G_6*2$\\
$[4,1;1,4]$ & Theorem~\ref{T3}, $n=2$\\
$[4,2;2,4]$ & sing $H(2,2)$\\
\bottomrule
\end{tabular}
\qquad
\begin{tabular}[t]{@{}l@{\quad}l@{}}
\toprule
\multicolumn{2}{@{}l}{$\rho=3$}\\
array & construction\\
\midrule
$[2,1,1;1,1,2]$ & $G_6+2$\\
$[4,2,2;2,2,4]$ & $G_6*2$\\
\bottomrule
\end{tabular}
\end{fitbox}\end{center}
\paragraph{$G_6^3$.} New constructions are marked with an asterisk.
\begin{center}\begin{fitbox}
\begin{tabular}[t]{@{}l@{\quad}l@{}}
\toprule
\multicolumn{2}{@{}l}{$\rho=1$}\\
array & construction\\
\midrule
$[2;1]$ & $G_6+4$\\
$[2;2]$ & $G_6+4$\\
$[4;2]$ & $G_6*2+2$\\
$[4;4]$ & $G_6^2+2$\\
$[6;2]$ & pc $H(3,2)$\\
$[6;3]$ & $G_6*3$\\
$[6;6]$ & even\\
\bottomrule
\end{tabular}
\qquad
\begin{tabular}[t]{@{}l@{\quad}l@{}}
\toprule
\multicolumn{2}{@{}l}{$\rho=2$}\\
array & construction\\
\midrule
$[1,1;1,1]$ & $G_6+4$\\
$[4,2;1,2]$ & $G_6^2+2$\\
$[2,2;2,2]$ & $G_6*2+2$\\
$[3,3;3,3]$ & $G_6*3$\\
$[4,1;1,4]$ & $G_6^2+2$\\
$[4,2;2,4]$ & $G_6^2+2$\\
$[6,1;2,6]$ & rep $H(3,3)$\\
$[6,2;4,6]$ & Theorem~\ref{T4}, $m=1$$^{\ast}$\\
\bottomrule
\end{tabular}
\qquad
\begin{tabular}[t]{@{}l@{\quad}l@{}}
\toprule
\multicolumn{2}{@{}l}{$\rho=3$}\\
array & construction\\
\midrule
$[2,1,1;1,1,2]$ & $G_6+4$\\
$[3,2,1;2,4,6]$ & sing $H(3,3)$\\
$[4,2,2;2,2,4]$ & $G_6*2+2$\\
$[6,3,3;3,3,6]$ & $G_6*3$\\
$[6,4,2;2,4,6]$ & sing $H(3,2)$\\
\bottomrule
\end{tabular}
\end{fitbox}\end{center}
\paragraph{$G_6^4$.} New constructions are marked with an asterisk.
\begin{center}
  \begin{fitbox}\begin{tabular}{ll@{\qquad}ll}
    \toprule
    array & construction & array & construction\\
    \midrule
    \multicolumn{2}{c}{\textit{Covering radius $\rho=1$}} & \multicolumn{2}{c}{\textit{Covering radius $\rho=2$}}\\
    \cmidrule(lr){1-2} \cmidrule(lr){3-4}
    $[2;\,1]$ & $G_6+6$ & $[1,1;\,1,1]$ & $G_6+6$\\
    $[2;\,2]$ & $G_6+6$ & $[2,1;\,2,4]$ & $G_6^2+4$\\
    $[4;\,2]$ & $G_6*2+4$ & $[2,2;\,2,2]$ & $G_6*2+4$\\
    $[4;\,4]$ & $G_6^2+4$ & $[3,3;\,3,3]$ & $G_6*3+2$\\
    $[5;\,4]$ & 4 pc $H(4,3)$ & $[4,1;\,1,4]$ & $G_6^2+4$\\
    $[6;\,2]$ & $G_6^3+2$ & $[4,2;\,2,4]$ & $G_6^2+4$\\
    $[6;\,3]$ & 3 pc $H(4,3)$ & $[4,2;\,4,8]$ & $G_6^2*2$\\
    $[6;\,6]$ & $G_6^3+2$ & $[4,4;\,4,4]$ & $G_6*4$\\
    $[7;\,2]$ & 2 pc $H(4,3)$ & $[5,2;\,2,5]$ & Corollary~\ref{Cor:linf}, $l=2$$^{\ast}$\\
    $[8;\,1]$ & pc $H(4,3)$ & $[6,1;\,2,6]$ & $G_6^3+2$\\
    $[8;\,4]$ & $G_6*4$ & $[6,2;\,4,6]$ & $G_6^3+2$ (Theorem~\ref{T4}, $m=1$)$^{\ast}$\\
    $[8;\,8]$ & even & $[7,1;\,1,7]$ & Corollary~\ref{Cor:cascade}, $l=2$$^{\ast}$\\
     &  & $[8,1;\,3,8]$ & Theorem~\ref{T3}, $n=4$$^{\ast}$\\
     &  & $[8,2;\,2,8]$ & $G_6^2*2$\\
     &  & $[8,2;\,6,8]$ & ext pc $H(4,2)$\\
     &  & $[8,4;\,4,8]$ & $G_6^2*2$\\
    \bottomrule
  \end{tabular}\end{fitbox}

\medskip
  \begin{fitbox}\begin{tabular}{ll@{\qquad}ll}
    \toprule
    array & construction & array & construction\\
    \midrule
    \multicolumn{2}{c}{\textit{Covering radius $\rho=3$}} & \multicolumn{2}{c}{\textit{Covering radius $\rho=4$}}\\
    \cmidrule(lr){1-2} \cmidrule(lr){3-4}
    $[2,1,1;\,1,1,2]$ & $G_6+6$ & $[4,3,2,1;\,2,4,6,8]$ & sing $H(4,3)$\\
    $[3,2,1;\,2,4,6]$ & $G_6^3+2$ & $[8,6,4,2;\,2,4,6,8]$ & sing $H(4,2)$\\
    $[4,2,2;\,2,2,4]$ & $G_6*2+4$ &  & \\
    $[6,3,3;\,3,3,6]$ & $G_6*3+2$ &  & \\
    $[6,4,2;\,2,4,6]$ & $G_6^3+2$ &  & \\
    $[8,4,4;\,4,4,8]$ & $G_6*4$ &  & \\
    $[8,7,1;\,1,7,8]$ & even pc $G_6^4$ &  & \\
    \bottomrule
  \end{tabular}\end{fitbox}
\end{center}
\subsection{$G_5^n$, $n\leq 4$}
\paragraph{$G_5$}
\begin{center}
\begin{tabular}[t]{@{}l@{\quad}l@{}}
\toprule
\multicolumn{2}{@{}l}{$\rho=2$}\\
array & construction\\
\midrule
$[1,1;1,2]$ & sing\\
\bottomrule
\end{tabular}
\end{center}
\paragraph{$G_5^2$}
\begin{center}\begin{fitbox}
\begin{tabular}[t]{@{}l@{\quad}l@{}}
\toprule
\multicolumn{2}{@{}l}{$\rho=1$}\\
array & construction\\
\midrule
$[4;1]$ & pc $G_5^2$\\
$[3;2]$ & 2 pc $G_5^2$\\
\bottomrule
\end{tabular}
\qquad
\begin{tabular}[t]{@{}l@{\quad}l@{}}
\toprule
\multicolumn{2}{@{}l}{$\rho=2$}\\
array & construction\\
\midrule
$[1,1;1,2]$ & $G_5+2$\\
$[2,2;2,4]$ & $G_5*2$\\
\bottomrule
\end{tabular}
\end{fitbox}\end{center}
\paragraph{$G_5^3$}
\begin{center}\begin{fitbox}
\begin{tabular}[t]{@{}l@{\quad}l@{}}
\toprule
\multicolumn{2}{@{}l}{$\rho=1$}\\
array & construction\\
\midrule
$[4;1]$ & $G_5^2+2$\\
$[3;2]$ & $G_5^2+2$\\
\bottomrule
\end{tabular}
\qquad
\begin{tabular}[t]{@{}l@{\quad}l@{}}
\toprule
\multicolumn{2}{@{}l}{$\rho=2$}\\
array & construction\\
\midrule
$[1,1;1,2]$ & $G_5+4$\\
$[2,2;2,4]$ & $G_5*2+2$\\
$[3,3;3,6]$ & $G_5*3$\\
\bottomrule
\end{tabular}
\end{fitbox}\end{center}
\paragraph{$G_5^4$}
\begin{center}\begin{fitbox}
\begin{tabular}[t]{@{}l@{\quad}l@{}}
\toprule
\multicolumn{2}{@{}l}{$\rho=1$}\\
array & construction\\
\midrule
$[4;1]$ & $G_5^2+4$\\
$[3;2]$ & $G_5^2+4$\\
$[6;4]$ & $G_5^2*2$\\
$[8;2]$ & $G_5^2*2$\\
\bottomrule
\end{tabular}
\qquad
\begin{tabular}[t]{@{}l@{\quad}l@{}}
\toprule
\multicolumn{2}{@{}l}{$\rho=2$}\\
array & construction\\
\midrule
$[1,1;1,2]$ & $G_5+6$\\
$[2,1;4,8]$ & K3 for a singleton in $H(2,5)$\\
$[2,2;2,4]$ & $G_5*2+4$\\
$[3,3;3,6]$ & $G_5*3+2$\\
$[4,2;3,6]$ & K3 for a cycle of length $4$ in $H(2,5)$ (Fig.~\ref{fig:H25})\\
$[4,4;4,8]$ & $G_5*4$\\
\bottomrule
\end{tabular}
\end{fitbox}\end{center}
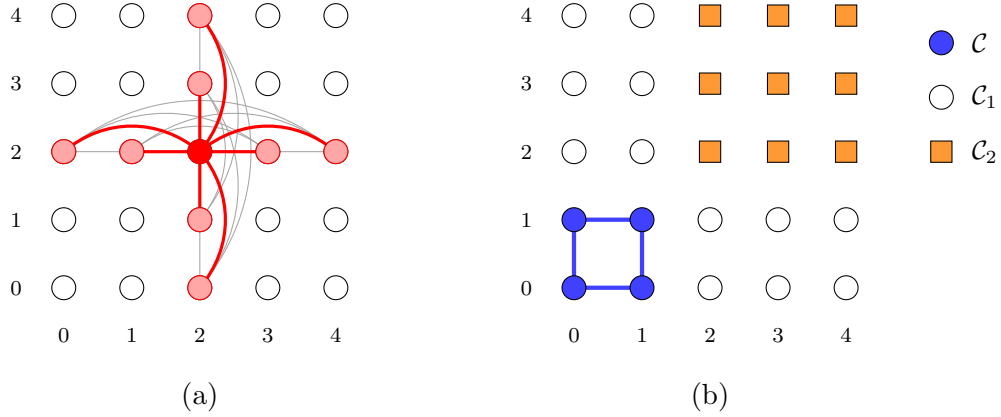
\begin{figure}[ht]
\centering
\begin{tikzpicture}[scale=0.9,
  vtx/.style={circle,draw=black,fill=white,inner sep=0pt,minimum size=9pt},
  cvtx/.style={circle,draw=black,inner sep=0pt,minimum size=9pt}]
\begin{scope}
  \foreach \a/\b in {0/1,1/2,2/3,3/4}{
    \draw[gray!70] (\a,2) -- (\b,2);
    \draw[gray!70] (2,\a) -- (2,\b);
  }
  \foreach \a/\b in {0/2,1/3,2/4,0/3,1/4,0/4}{
    \draw[gray!70] (\a,2) to[bend left=40] (\b,2);
    \draw[gray!70] (2,\a) to[bend right=40] (2,\b);
  }
  \draw[red,very thick] (1,2) -- (3,2);
  \draw[red,very thick] (2,1) -- (2,3);
  \draw[red,very thick] (0,2) to[bend left=40] (2,2) to[bend left=40] (4,2);
  \draw[red,very thick] (2,0) to[bend right=40] (2,2) to[bend right=40] (2,4);
  \foreach \i in {0,...,4}{\foreach \j in {0,...,4}{
    \node[vtx] at (\i,\j) {};
  }}
  \foreach \i in {0,1,3,4}{
    \node[vtx,fill=red!35,draw=red] at (\i,2) {};
    \node[vtx,fill=red!35,draw=red] at (2,\i) {};
  }
  \node[vtx,fill=red,draw=red,minimum size=9pt] at (2,2) {};
  \foreach \i in {0,...,4}{
    \node[font=\scriptsize] at (\i,-0.7) {$\i$};
    \node[font=\scriptsize] at (-0.7,\i) {$\i$};
  }
  \node at (2,-1.6) {(a)};
\end{scope}
\begin{scope}[xshift=7.5cm]
  \draw[blue!75,line width=1.6pt] (0,0) -- (1,0) -- (1,1) -- (0,1) -- cycle;
  \foreach \i in {2,3,4}{\foreach \j in {2,3,4}{
    \node[cvtx,fill=orange!80,rectangle,minimum size=8pt] at (\i,\j) {};
  }}
  \foreach \i in {0,1}{\foreach \j in {2,3,4}{
    \node[cvtx,fill=white] at (\i,\j) {};
    \node[cvtx,fill=white] at (\j,\i) {};
  }}
  \foreach \i in {0,1}{\foreach \j in {0,1}{
    \node[cvtx,fill=blue!75] at (\i,\j) {};
  }}
  \foreach \i in {0,...,4}{
    \node[font=\scriptsize] at (\i,-0.7) {$\i$};
    \node[font=\scriptsize] at (-0.7,\i) {$\i$};
  }
  \node[cvtx,fill=blue!75] at (5.4,3.6) {};
  \node[anchor=west,font=\small] at (5.7,3.6) {$\mathcal C$ };
  \node[cvtx,fill=white] at (5.4,2.8) {};
  \node[anchor=west,font=\small] at (5.7,2.8) {${\mathcal C}_1$ };
  \node[cvtx,fill=orange!80,rectangle,minimum size=8pt] at (5.4,2.0) {};
  \node[anchor=west,font=\small] at (5.7,2.0) {${\mathcal C}_2$ };
  \node at (2,-1.6) {(b)};
\end{scope}
\end{tikzpicture}
\caption{(a) The adjacency in the Hamming graph  $H(2,5)$ and (b) CRC  ${\mathcal C} =\{0,1\}^2$ with parameter matrix $\left(\begin{smallmatrix}2&6&0\\2&3&3\\0&4&4\end{smallmatrix}\right)$}
\label{fig:H25}
\end{figure}

\subsection{$G_7^n$, $n\leq 4$}
\paragraph{$G_7$}
\begin{center}
\begin{tabular}[t]{@{}l@{\quad}l@{}}
\toprule
\multicolumn{2}{@{}l}{$\rho=3$}\\
array & construction\\
\midrule
$[1,1,1;1,1,2]$ & sing\\
\bottomrule
\end{tabular}
\end{center}
\paragraph{$G_7^2$}
\begin{center}
\begin{tabular}[t]{@{}l@{\quad}l@{}}
\toprule
\multicolumn{2}{@{}l}{$\rho=3$}\\
array & construction\\
\midrule
$[1,1,1;1,1,2]$ & $G_7+2$\\
$[2,2,2;2,2,4]$ & $G_7*2$\\
\bottomrule
\end{tabular}
\end{center}
\paragraph{$G_7^3$}
\begin{center}\begin{fitbox}
\begin{tabular}[t]{@{}l@{\quad}l@{}}
\toprule
\multicolumn{2}{@{}l}{$\rho=1$}\\
array & construction\\
\midrule
$[6;1]$ & pc $G_7^3$\\
$[5;2]$ & 2 pc $G_7^3$\\
$[4;3]$ & 3 pc $G_7^3$\\
\bottomrule
\end{tabular}
\qquad
\begin{tabular}[t]{@{}l@{\quad}l@{}}
\toprule
\multicolumn{2}{@{}l}{$\rho=3$}\\
array & construction\\
\midrule
$[1,1,1;1,1,2]$ & $G_7+4$\\
$[2,2,2;2,2,4]$ & $G_7*2+2$\\
$[3,3,3;3,3,6]$ & $G_7*3$\\
\bottomrule
\end{tabular}
\end{fitbox}\end{center}
\paragraph{$G_7^4$}
\begin{center}\begin{fitbox}
\begin{tabular}[t]{@{}l@{\quad}l@{}}
\toprule
\multicolumn{2}{@{}l}{$\rho=1$}\\
array & construction\\
\midrule
$[6;1]$ & $G_7^3+2$\\
$[5;2]$ & $G_7^3+2$\\
$[4;3]$ & $G_7^3+2$\\
\bottomrule
\end{tabular}
\qquad
\begin{tabular}[t]{@{}l@{\quad}l@{}}
\toprule
\multicolumn{2}{@{}l}{$\rho=3$}\\
array & construction\\
\midrule
$[1,1,1;1,1,2]$ & $G_7+6$\\
$[2,2,2;2,2,4]$ & $G_7*2+4$\\
$[3,3,3;3,3,6]$ & $G_7*3+2$\\
$[4,4,4;4,4,8]$ & $G_7*4$\\
\bottomrule
\end{tabular}
\end{fitbox}\end{center}

\bigskip

{\bf Acknowledgments} The author is grateful to S.~V.~Avgustinovich for  discussions and pointing out the connection between  construction of the code of length \(3\) given in Section 3.2, to A.~Yu.~Vasil'eva for fruitful discussion of the material of the paper,  to V.~R.~Danilko for kindly providing multicore computing resources for solving the binary linear programming problems and A.~A.~Panin for recommending the HiGHS package for Python.

The results of the paper  were obtained in the pre-AI era. At the final part of the paper production the proofs  were cross-verified in the Lean language using the AI Aristotle agent: \href{https://aristotle.harmonic.fun}{aristotle.harmonic.fun}, which was also used to draw all the figures and  reformat raw data into tables.

\end{document}